\documentclass{amsart}
\usepackage{amsthm}
\usepackage{amsmath}
\usepackage{amssymb}

\DeclareMathOperator{\RP}{\mathbb{R}P}
\DeclareMathOperator{\deRham}{H}
\DeclareMathOperator{\Symp}{Symp}
\DeclareMathOperator{\Ham}{Ham}
\DeclareMathOperator{\Flux}{Flux}
\DeclareMathOperator{\id}{id}
\DeclareMathOperator{\C}{\mathbb{C}}
\DeclareMathOperator{\R}{\mathbb{R}}
\DeclareMathOperator{\Z}{\mathbb{Z}}

\makeatletter
	
	\@addtoreset{equation}{section}
	\makeatother
\theoremstyle{definition}
\newtheorem{Defn}{Definition}[section]

\theoremstyle{plain}
\newtheorem{Lemm}[Defn]{Lemma}

\newtheorem{Theo}[Defn]{Theorem}

\theoremstyle{remark}

\newtheorem{Example}[Defn]{Example}

\begin{document}

	\author[H. Ishida]{Hiroaki Ishida}
	\address{Graduate~School~of~Science, Osaka~City~University, Sugimoto, Sumiyoshi-ku, Osaka~558-8585, Japan}
	\email{hiroaki.ishida86@gmail.com}
	\title{Symplectic real Bott manifolds}
	\date{\today}
	\keywords{toric topology, symplectic topology, real Bott manifold}
	\subjclass[2000]{57R17,57S25}
\begin{abstract}
A real Bott manifold is the total space of an iterated $\RP ^1$-bundles over a point, where each $\RP^1$-bundle is the projectivization of a Whitney sum of two real line bundles. In this paper, we characterize real Bott manifolds which admit a symplectic form. In particular, it turns out that a real Bott manifold admits a symplectic form if and only if it is cohomologically symplectic. In this case, it admits even a K\"{a}hler structure. We also prove that any symplectic cohomology class of a real Bott manifolds can be represented by a symplectic form. Finally, we study the flux of a symplectic real Bott manifold. 
\end{abstract}

\maketitle

\section{Introduction}
A \emph{real Bott tower} (of height $n$) is a sequence of $\RP ^1$-bundles:
\begin{equation*}
	 M_n \to M_{n-1} \to \dots \to M_1 \to M_0 =\{ \text{a point}\} ,
\end{equation*}
where each $\RP^1$-bundle $M_i \to M_{i-1}$ is the projectivization of a Whitney sum of two real line bundles on $M_{i-1}$. Each $M_i$ is called a \emph{real Bott manifold}. Clearly $M_1 = \RP ^1$ and $M_2 = (\RP ^1)^2 \text{ or a Klein bottle}$. If every bundle in the tower is trivial, then $M_n = (\RP ^1)^n$. However, there are many choices of non-trivial bundles at each stage in the tower and it is known that there are many different diffeomorphism classes in real Bott manifolds (\cite{KM09}, \cite{Masuda}).  A real Bott manifold is also an example of a real toric manifold which admits a flat Riemannian metric (\cite{KM09}). 

Although orientable ones occupy a small portion in all real Bott manifolds (\cite{choi09}), the number of orientable ones of dimension $n$ approaches infinity as $n$ approaches infinity. Among those orientable ones, some are \emph{symplectic}, i.e., admit a symplectic form. In this paper we give a complete characterization of symplectic real Bott manifolds (Theorem~\ref{Theo:maintheo}).  In particular, we prove that among real Bott manifolds $M$ the following are equivalent:
\begin{enumerate}
\item[(1)] $M$ is cohomologically symplectic,
\item[(2)] $M$ is symplectic,
\item[(3)] $M$ admits a K\"ahler structure.  
\end{enumerate}
We remark that the implication (3) $\Rightarrow$ (2) $\Rightarrow$ (1) always holds but the reverse implications (1) $\Rightarrow$ (2) and (2) $\Rightarrow$ (3) do not hold in general as is well-known. 
For example, $\C P^2\#\C P^2$ is cohomologically symplectic but not symplectic because it does not admit an almost complex structure and a certain $T^2$-bundle over $T^2$ constructed in \cite{thur76} is symplectic but does not admit a K\"ahler structure.  

This paper is organized as follows.  In Section~\ref{sec:quotientconst} we recall the quotient description of real Bott manifolds. In Section~\ref{sec:Main_results} we state and prove our main theorem.  In Section \ref{sec:The_flux_group} we study the flux group of a symplectic real Bott manifold. 

Throughout this paper, all cohomology will be de Rham cohomology over $\R$.  

\section{Quotient description of real Bott manifolds}\label{sec:quotientconst}
In this section, we recall the quotient description of real Bott manifolds (see \cite{KM09} and \cite{Masuda} for details) and observe the cohomology ring of a real Bott manifold.  

Let $\mathfrak{B}(n)$ be the the set of $n \times n$ upper triangular $(0,1)$ matrices with zero diagonal entries. 
For a matrix $A \in \mathfrak{B}(n)$, $A^i_j$ denotes the $(i,j)$ entry of $A$ 
and $A^i$ (respectively, $A_j$) denotes the $i$-th row (respectively, $j$-th column) of $A$. Let $S^1$ be the unit circle in $\C$. For $z \in S^1$ and $a \in \Z /2 = \{ 0,1\}$, 
we set $z(a):=a$ if $a=0$ and $\bar{z}$ if $a=1$. 
%we use the following notation:
%\begin{equation*}
%	z(a):= \begin{cases}
%			z & \text{ if } a=0, \\
%			\bar{z} & \text{ if } a=1.
%	\end{cases}
%\end{equation*}
We then define the involution $a_i$ on $T^n := (S^1)^n$ by 
\begin{equation}\label{eq:actionontori}
a_i(z_1,\dots ,z_n):= (z_1,\dots ,z_{i-1},-z_i,z_{i+1}(A_{i+1}^i),\dots ,z_n(A_n^i))
\end{equation}
for $i=1,\dots ,n$. Let $G(A)$ denote the transformation group on $T^n$ generated by $a_i$'s. 
Then the quotient space $M(A):= T^n/G(A)$ is known to be a real Bott manifold and every real Bott manifold can be obtained as $M(A)$ for some $A \in \mathfrak{B}(n)$. Although $A$ is not necessarily uniquely determined by a real Bott manifold, $A$ contains all geometrical information on $M(A)$. For example, 
\begin{equation} \label{eq:orientability}
M(A) \text{ is orientable } \Longleftrightarrow \sum _{j=1}^nA^i_j =0 \text{ in } \Z /2 \text{ for any } i
\end{equation}
(see \cite{KM09}). 

It is also helpful to describe $M(A)$ as the quotient of $\R^n$ by affine transformations.  In fact, let $\Gamma (A)$ denote the affine transformation group on $\R ^n$ generated by $s_i$'s defined by 
\begin{equation}\label{eq:actiononEuclidean}
	s_i(u_1,\dots ,u_n):= (u_1,\dots ,u_{i-1}, u_i+ \frac{1}{2},(-1)^{A^i_{i+1}}u_{i+1}, \dots ,(-1)^{A_n^i}u_n)
\end{equation}
for $i=1,\dots ,n$. Then, an exponential map from $\R$ to $S^1$ sending $u$ to $\exp(2\pi\sqrt{-1}u)$ induces a diffeomorphism from $\R ^n/\Gamma (A)$ onto $T^n/G(A)=M(A)$. 

Let $du_1,\dots ,du_n$ denote the standard $1$-forms on $\R^n$. 
Since each $du_j$ is invariant under parallel translations on $\R ^n$, 
it descends to a closed $1$-form on $T^n \cong \R ^n/\Z ^n$, which we also denote by $du_j$.  The (de Rham) cohomology ring $\deRham^*(T^n)$ of $T^n$ is the exterior algebra in $n$ variables $[du_1], \dots ,[du_n]$ over $\R$, 
where $[du_j]$ denotes the cohomology class represented by the $1$-form $du_j$. It follows from \eqref{eq:actionontori} or \eqref{eq:actiononEuclidean} that the endomorphism $a_i^*$ of $\deRham^*(T^n)$ induced by $a_i\in G(A)$ is given by 
\begin{equation} \label{eq:actionondeRham} 
		a_i^*([du_j]) = \begin{cases}
			[du_j] & \text{ if $A^i_j=0$},\\
			-[du_j] & \text{ if $A^i_j=1$}.
		\end{cases}
\end{equation}
We note that since $M(A)=T^n/G(A)$ and $G(A)$ is a finite group, we have 
\begin{equation} \label{eq:quotient}
\deRham^*(M(A))=\deRham^*(T^n)^{G(A)}
\end{equation}
(see \cite[Theorem 2.4 in p.120]{Bred72} for example), where the right hand side denotes the $G(A)$-invariants in $\deRham^*(T^n)$. 

\begin{Lemm}\label{Lemm:invariance}
Let $J$ be a subset of $\{ 1,\dots ,n\}$. Then $\prod _{j\in J}[du_j]\in \deRham^*(T^n)$ is $G(A)$-invariant if and only if $\sum _{j \in J}A_j=0$ in $\Z/2$. 
\end{Lemm}

\begin{proof}
	By \eqref{eq:actionondeRham}, we have 
	\begin{equation*}
		a_i^*(\prod _{j\in J}[du_j])=(-1)^{\sum _{j \in J}A^i_j}\prod _{j\in J}[du_j],
	\end{equation*}
	Thus, $\prod _{j\in J}[du_j]$ is fixed by $a_i^*$ if and only if $\sum _{j \in J}A^i_j =0 \text{ in } \Z /2$. This implies the lemma since $G(A)$ is generated by $a_i$'s. 
\end{proof}

\section{Main theorem}\label{sec:Main_results}

The following is our main theorem in this paper.

\begin{Theo}\label{Theo:maintheo}
Let $A \in \mathfrak{B}(2n)$. The following conditions are equivalent:
\begin{enumerate}
\item[(1)] $M(A)$ is cohomologically symplectic, that is, there exists an $\alpha \in \deRham ^2(M(A))$ such that $\alpha ^n$ is nonzero. 
\item[(2)] There exist $n$ subsets $\{j_1,j_{n+1}\}, \dots ,\{j_n,j_{2n}\}$ of $\{1,2,\dots,2n\}$ such that 
	\begin{itemize}
		\item $\coprod _k^n\{ j_k,j_{k+n}\} = \{1,2,\dots,2n\}$ and 
		\item $A_{j_1}=A_{j_{n+1}}, \dots, A_{j_n}=A_{j_{2n}}$.
	\end{itemize}
\item[(3)] There exists a symplectic form on $M(A)$. 
\item[(4)] There exists a K\"{a}hler structure on $M(A)$.
\end{enumerate}
Moreover, any $\alpha \in \deRham^2(M(A))$ in {\rm (1)} can be represented by a symplectic form on $M(A)$. 
\end{Theo}

\begin{proof}
Because any closed symplectic manifold is cohomologically symplectic and any K\"{a}hler manifold is a symplectic manifold, it suffices to prove implications (1) $\Rightarrow$ (2) and (2) $\Rightarrow$ (4). 

Proof of (1) $\Rightarrow$ (2).  
Assume that there exists a de Rham cohomology class $\alpha \in \deRham ^2(M(A))$ such that $\alpha ^n \neq 0$. We identify $\deRham^*(M(A))$ with $\deRham^*(T^n)^{G(A)}$ by \eqref{eq:quotient}.  Then it follows from Corollary \ref{Lemm:invariance} that we can write $\alpha$ uniquely as 
\begin{equation}\label{eq:alpha}
\alpha = \sum _{j<k, A_j=A_k}c_{j,k}[du_j\wedge du_k]\quad
\text{with some $c_{j,k} \in \R$.}
\end{equation}
Thus $\alpha ^n \neq 0$ implies the condition (2). 

Proof of (2) $\Rightarrow$ (4). 
Assume that $A \in \mathfrak{B}(2n)$ satisfies the condition (2), namely 
$A_{j_k}=A_{j_{k+n}}$ for $k=1,\dots ,n$. Then we identify $\R ^{2n}$ with 
$\C ^n$ by 
\begin{equation*}
		z_k:=u_{j_k}+\sqrt{-1}u_{j_{k+n}}
\end{equation*}
for $k=1,\dots ,n$. Consider the standard Hermitian metric on $\C ^n$. Then, 
$\Gamma (A)$ acts on $\C ^n$ as biholomorphisms and isometries. In fact, 
through the above identification, it follows from (\ref{eq:actiononEuclidean}) 
that the action of $s_i \in \Gamma (A)$ on $\C^n$ is given by 
\begin{equation*}
		s_i(z_1,\dots ,z_n)_k= 
		\begin{cases}
	z_k+\frac{1}{2} & \text{ if $i=j_k$}, \\
	z_k+\frac{\sqrt{-1}}{2} & \text{ if $i=j_{k+n}$}, \\
	z_k & \text{ if $A^i_{j_k}=A^i_{j_{k+n}}=0$ and $i\neq j_k,j_{k+n}$}, \\
	-z_k & \text{ if $A^i_{j_k}=A^i_{j_{k+n}} =1$},
		\end{cases}
\end{equation*}
where the left hand side denotes the $k$-th component of 
$s_i(z_1,\dots ,z_n)$. Thus the quotient $M(A)=\C ^n/\Gamma (A)$ inherits the standard K\"{a}hler structure on $\C^n$.  

Finally, we shall prove the last statement in the theorem. As 
observed above, $\alpha\in \deRham^2(M(A))$ is of the form \eqref{eq:alpha}.  
We then define the differential closed $2$-form $\omega$ on $\R ^{2n}$ by 
\begin{equation} \label{eq:omegaMA}
\omega := \sum _{j<k, A_j=A_k}c_{j,k}du_j \wedge du_k.
\end{equation}
Comparing \eqref{eq:alpha} with \eqref{eq:omegaMA}, one sees that the 
condition $\alpha^n \neq 0$ implies that $\omega^n$ is nowhere zero. 
Thus $\omega$ is a symplectic form on 
$\R ^{2n}$. Since $\omega $ is invariant under the $\Gamma (A)$-action on 
$\R ^{2n}$, $\omega$ descends to a symplectic form on the quotient 
$M(A)=\R ^{2n}/\Gamma (A)$ and this represents the given class $\alpha$. 
\end{proof}

\begin{Example}
Let $A\in \mathfrak{B}(4)$. If $A$ is the zero matrix, then $M(A)$ is the 4-dimensional torus and symplectic.  Suppose that $A$ is non-zero and $M(A)$ is symplectic.  Then it follows from Theorem~\ref{Theo:maintheo} (2) that $A$ is one of the following: 
\[
{\tiny 
\begin{pmatrix}
0&1&1&0\\
0&0&0&0\\
0&0&0&0\\
0&0&0&0
\end{pmatrix},\quad 
\begin{pmatrix}
0&1&0&1\\
0&0&0&0\\
0&0&0&0\\
0&0&0&0
\end{pmatrix},\quad 
\begin{pmatrix}
0&0&1&1\\
0&0&0&0\\
0&0&0&0\\
0&0&0&0
\end{pmatrix},\quad
\begin{pmatrix}
0&0&1&1\\
0&0&1&1\\
0&0&0&0\\
0&0&0&0
\end{pmatrix},\quad 
\begin{pmatrix}
0&0&0&0\\
0&0&1&1\\
0&0&0&0\\
0&0&0&0
\end{pmatrix}.
}
\]
Real Bott manifolds $M(A)$ for $A$ above are diffeomorphic to each other but not diffeomorphic to the 4-dimensional torus (\cite{KM09}, \cite{Masuda}). One sees that $M(A)$ is the total space of a non-trivial $T^2$-bundle over $T^2$.  On the other hand, $T^2$-bundles over $T^2$ which are symplectic are classified in \cite{Geiges92}. One can easily check that our $M(A)$ is of type $\{ -I,I,(0,0)\}$ in \cite[Table 1]{Geiges92} . 

Finally we note that if 
\begin{equation*}
A= 
\tiny 
\begin{pmatrix}
0&1&1&0\\
0&0&1&1\\
0&0&0&0\\
0&0&0&0
\end{pmatrix},
\end{equation*}
then $M(A)$ is orientable by \eqref{eq:orientability}, but not symplectic. 
Therefore the class of symplectic real Bott manifolds is strictly smaller than that of orientable real Bott manifolds. 
\end{Example}

\section{The flux group}\label{sec:The_flux_group}
In this section, we will study the flux group of a symplectic real Bott 
manifold.  For that, we recall the definition of a flux group for a general symplectic manifold. 

Let $(M,\omega )$ be a closed symplectic manifold. 
A diffeomorphism $\phi : M \to M$ is called a \emph{symplectomorphism} if $\phi ^*\omega = \omega$ 
and the group of symplectomorphisms of $(M,\omega )$ is denoted by $\Symp (M,\omega )$. 
Associated to a smooth function $f : M \to \R$, the Hamiltonian vector field $X_f$ is defined by 
$i_{X_f}\omega =df$. For a one-parameter family $\{ f_t\} _{ 0\leq t \leq 1}$ of functions, 
we obtain a one parameter family $\{ X_{f_t}\} _{0 \leq t \leq 1}$ of Hamiltonian vector fields, 
and integrating $\{ X_{f_t}\}$, we obtain a one-parameter family $\{ \phi _t\} _{0 \leq t \leq 1}$ of 
diffeomorphisms defined by 
\begin{equation*}
	\frac{d}{dt}\phi _t=X_{f_t}\circ \phi _t\ \text{ and }\ \phi _0 = \id.
\end{equation*}
The time-one map $\phi _1$ is a symplectomorphism and called a \emph{Hamiltonian diffeomorphism}. 
It is known that all Hamiltonian diffeomorphisms of $(M,\omega )$ form a subgroup, denoted $\Ham (M,\omega )$,
of the identity component $\Symp _0(M,\omega )$ of $\Symp (M,\omega )$. For a symplectic isotopy $\{ \phi _t\}$,
that is, an isotopy through symplectomorphisms, 
we obtain a one-parameter family $\{ X_t\}$ of vector fields define by 
\begin{equation*} 
	\frac{d}{dt}\phi _t = X_t \circ \phi _t.
\end{equation*}
The \emph{flux} of $\{ \phi _t\}$ is then defined to be 
\begin{equation}\label{eq:flux}
	\int _0^1[i_{X_t}\omega ]dt \in \deRham ^1(M).
\end{equation}
It is known that the flux depends only on the homotopy class of symplectic isotopies 
with fixed end points $\phi _0 = \id $ and $\phi _1$, so that it defines a homomorphism 
\begin{equation*}
	\Flux : \Symp _0(M,\omega ) \to \deRham ^1(M)/\Gamma _\omega ,
\end{equation*}
where $\Gamma _\omega$ is the image of the fundamental group $\pi _1(\Symp _0(M,\omega ))$ by the 
flux and called the \emph{flux group} of $(M,\omega )$. 
The solution of the flux conjecture (\cite{Ono06}) says that the subgroup $\Gamma _\omega$ of
$\deRham ^1(M)$ is closed and discrete. 
According to \cite{Banyaga78}, the kernel of $\Flux$ is exactly equal to $\Ham (M,\omega )$, 
in other words, we have an exact sequence
\begin{equation*}
	\{ 1\} \to \Ham (M,\omega ) \to \Symp _0(M,\omega ) \overset{\Flux}{\to} \deRham ^1(M)/\Gamma _\omega . 
\end{equation*}

Now, we consider the flux of a symplectic real Bott manifold. 

\begin{Theo}
Let $M(A)$ be a real Bott manifold with a symplectic form $\omega$ given by 
\eqref{eq:omegaMA}. Then, the flux group $\Gamma _\omega $ is a lattice group 
of $\deRham ^1(M(A))$ of full rank. 
\end{Theo}

\begin{proof}
It follows from Lemma~\ref{Lemm:invariance} that $\deRham^1(M(A))$ is generated by $[du_j]$ with $A_j =0$, and since $M(A)$ is symplectic, the number of zero columns in $A$ is even by Theorem \ref{Theo:maintheo}, so that $\deRham^1(M(A))$ is even dimensional. Let $2r$ be the dimension of $\deRham ^1(M(A))$. We may assume that $\deRham^1(M(A))$ is generated by $du_1,\dots ,du_{2r}$ by changing the suffices of the coordinates. Moreover, through a linear coordinate change of the first $2r$ coordinates $u_1,\dots ,u_{2r}$, we may assume that the symplectic form $\omega $ on $M(A)$ is of the form 
\begin{equation}\label{eq:omega}
\omega = \sum _{i=1}^rdu_i\wedge du_{i+r} + \sum _{j<k,A_j=A_k\neq 0}c_{j,k}du_j\wedge du_k.
\end{equation}
	
Since $M(A) = T^{2n}/G(A)$ and $A_p=0$ for $p=1,\dots ,2r$, the multiplication of $S^1$ on the $p$-th coordinate on $T^{2n}$ for $1 \leq p \leq 2r$ descends to an $S^1$-action on $M(A)$ and defines a symplectic isotopy $\{ \phi _t^p\}$. The one-parameter family $\{ X_t^p\}$ of vector fields associated with $\{ \phi _t^p\}$ is then $\partial /\partial u_p$ (possibly up to a non-zero constant), so that it 	follows from (\ref{eq:flux}) and (\ref{eq:omega}) that 
\begin{equation*}
\text{the flux of } \{ \phi ^p_t\} = \int ^1_0[i_{X_t^p}\omega ]dt = \int _0^1[du_q]dt = [du_q]
\end{equation*}
where $q= p+r$ if $1 \leq p \leq r$ and $q= p-r $ if $r+1 \leq p \leq 2r$. 
This shows that $\Gamma _\omega $ spans $\deRham ^1(M(A))$ over $\R$. Since $\Gamma _\omega$ is closed and discrete in $\deRham^1(M(A))$ as remarked before, it must be a lattice group of $\deRham^1(M(A))$ of full rank. 
\end{proof}

\bigskip
\noindent
{\bf Acknowledgment}. The author would like to thank Professor Mikiya Masuda for stimulating discussion about
toric topology and symplectic topology.  
He also would like to thank Professor Kaoru Ono for explaining the flux group of a symplectic manifold, Professor Takahiko Yoshida for informing on the paper \cite{Geiges92}, and Shintaro Kuroki and Yunhyung Cho for useful comments on the paper.

\end{document}